\documentclass[12pt]{amsart}

\advance\hoffset by -0.5 truecm

\usepackage{mathrsfs}

\usepackage{amssymb}
\usepackage{amscd}
\usepackage{verbatim}
\usepackage{epsfig}
\usepackage{euscript}
\usepackage[colorlinks=true]{hyperref}
\hypersetup{urlcolor=blue, citecolor=red}

\newtheorem{Theorem}{Theorem}[section]

\newtheorem{Lemma}[Theorem]{Lemma}
\newtheorem{Proposition}[Theorem]{Proposition}
\theoremstyle{definition}

\newtheorem{Remark}[Theorem]{Remark}

\def \R{\mathbb R}
\def \p{\partial}

\def\g{\gamma}
\def \g{\gamma}

\def \<{\langle}
\def \>{\rangle}

\def \tX{\widetilde X}
\def \tx{\tilde x}
\def \e{\epsilon}
\def \o{\omega}
\def \l{\lambda}
\def \S{\mathbb S}
\def \gxy{\gamma_{x,y,\lambda,\omega}}
\def \xy{x,y,\lambda,\omega}
\def \tc{\tilde{c}}
\def \tg{\tilde{g}}
\def \tL{\tilde{L}}
\def \tl{\tilde{\ell}}
\def \tO{\tilde{\Omega}}
\def \tV{\tilde{V}}
\def \X0{X^{(0)}}
\def \tX{\tilde{X}}
\def \s{\sigma}
\def \S{\Sigma}
\def \tY{\tilde{Y}}
\def \tga{\tilde{\gamma}}

\def \tM{\widetilde M}

\newcommand{\supp}{\operatorname{supp}}

\newcommand{\diag}{\operatorname{diag}}

\begin{document}

\title[Lens rigidity with partial data in the presence of a magnetic field]{Lens rigidity with partial data in the presence of a magnetic field}

\author[Hanming Zhou]{Hanming Zhou}
\address{Department of Pure Mathematics and Mathematical Statistics, University of Cambridge, Cambridge CB3 0WB, UK}
\email{hz318@dpmms.cam.ac.uk}

%\date{\today}

%\thanks{This research was supported by EPSRC grant EP/M023842/1}

\begin{abstract}
In this paper we consider the lens rigidity problem with partial data for conformal metrics in the presence of a magnetic field on a compact manifold of dimension $\geq 3$ with boundary. We show that one can uniquely determine the conformal factor and the magnetic field near a strictly convex (with respect to the magnetic geodesics) boundary point where the lens data is accessible. We also prove a boundary rigidity result with partial data assuming the lengths of magnetic geodesics joining boundary points near a strictly convex boundary point are known. The local lens rigidity result also leads to a global rigidity result under some strictly convex foliation condition. A discussion of a weaker version of the lens rigidity problem with partial data for general smooth curves is given at the end of the paper.
\end{abstract}
\maketitle

\section{Introduction and main results}
Given a Riemannian manifold $(M, g)$, endowed with a magnetic field $\Omega$, that is a closed 2-form, we consider the law of motion described by 
\begin{equation}\label{Mag}
\nabla_{\dot\gamma}\dot\gamma=Y(\dot\gamma),
\end{equation}
where $\nabla$ is the Levi-Civita connection of $g$ with the Christoffel symbols $\{\Gamma^i_{jk}\}$ and $Y:TM\to TM$ is the \emph{Lorentz force} associated with $\Omega$, i.e. the bundle map uniquely determined by
\begin{equation}\label{lorentz}
\Omega_z(v, w)=\langle Y_z(v), w\rangle_g
\end{equation}
for all $z\in M$ and $v, w\in T_zM$. A curve $\gamma:\R\to M$, satisfying \eqref{Mag} is called a \emph{magnetic geodesic}. The equation \eqref{Mag} defines a flow $\phi_t: t\rightarrow (\g(t), \dot\g(t))$ on $TM$ that we call a \emph{magnetic flow}. It is not difficult to show that the generator ${\bf G}_{\mu}$ of the magnetic flow is 
\begin{equation}\label{generator}
{\bf G}_{\mu}(z, v)={\bf G}(z,v)+Y^j_i(z) v^i\frac{\p}{\p v^j},
\end{equation}
where ${\bf G}(z, v)=v^i\frac{\p}{\p x^i}-\Gamma^i_{jk} v^j v^k\frac{\p}{\p v^i}$ is the generator of the geodesic flow. Note that time is not reversible on the magnetic geodesics, unless $\Omega=0$. When $\Omega=0$ we obtain the ordinary geodesic flow. We call the triple $(M, g, \Omega)$ a \emph{magnetic system}. 

It turns out that the magnetic flow is the Hamiltonian flow of $H(v)=\frac{1}{2}|v|_g^2, \, v\in TM$ with respect to the symplectic form $\beta=\beta_0+\pi^*\Omega$, where $\beta_0$ is the canonical symplectic form on $TM$ and $\pi: TM\rightarrow M$ is the canonical projection. %Therefore, the magnetic flow describes the motion of a particle on a Riemannian manifold under the influence of a magnetic field $\Omega$. 
Magnetic flows and magnetic geodesics were first considered in \cite{Ar61, AS67}. %and it was shown in \cite{ArG, Ko, N1, N2, NS, PP} that they are related to dynamical systems, symplectic geometry, classical mechanics and mathematical mechanics.
Since the magnetic flow preserves the level sets of the Hamiltonian function $H$, every magnetic geodesic has constant speed. %Unlike the geodesic flow, where the flow is the same (up to time scale) on any energy levels, the magnetic flow depends essentially on the choice of the energy level.
 In the current paper we fix the energy level $H^{-1}(\frac{1}{2})$, %thus considering the magnetic flow on the unit sphere bundle $SM$ of $M$,
i.e. we only consider the unit speed magnetic geodesics. %However, fixing the energy level to be $SM$ is no restriction at all, one can obtain the behavior in any energy level by considering the flow on $SM$ upon changing $\Omega$ by $\kappa\Omega$, where $\kappa\in\R$.

From now on, we assume that $\p M\neq \emptyset$. Let $z\in\p M$, $S\p M$ be the unit sphere bundle of the boundary $\p M$, we say $M$ is \emph{strictly magnetic convex} at $z$ if 
$$\Lambda(z,v)>\<Y_z(v), \nu(z)\>_g$$ 
for all $v\in S_z\p M$, where $\Lambda$ is the second fundamental form of $\p M$, $\nu(z)$ is the inward unit vector normal to $\p M$ at $z$. When $Y=0$, this is consistent with the ordinary definition of convexity.%If the inequality is strict, then we say $M$ is \emph{strictly magnetic convex} at $z$. 

In this paper, we mainly consider the lens rigidity problem for magnetic systems. To define the lens data of a magnetic system, we introduce the manifolds
$$\p_{\pm}SM=\{(z,v):\, z\in\p M, v\in S_zM, \pm (v, \nu(z))\geq 0\}.$$
We define the \emph{scattering relation}
$$L: \p_+SM\rightarrow \p_-SM$$
as follows: given $(z, v)\in\p_+SM, \, L(z,v)=(z', v')$, where $z'$ is the exit point, $v'$ is the exit direction, if exist, of the maximal unit speed magnetic geodesic $\g_{z,v}$ issued from $(z,v)$. Let 
$$\ell: \p_+SM\rightarrow \R\cup\infty$$
be the \emph{travel time}, which is the length of $\g_{z,v}$, possibly infinite. If $\ell<\infty$, we say the magnetic system $(M,g,\Omega)$ is \emph{non-trapping}. $L$ and $\ell$ together are called the \emph{lens data} of the magnetic system. It is easy to check that given a diffeomorphism $\psi: M\to M$ fixing the boundary, the magnetic systems $(M,g,\Omega)$ and $(M,\psi^*g, \psi^*\Omega)$ have the same lens data.
The {\em lens rigidity} problem for magnetic systems asks whether the lens data $(L, \ell)$ uniquely determines a magnetic system up to the natural obstruction above.

The lens rigidity problem is closely related to another rigidity problem, namely the boundary rigidity problem. The latter, also known as the travel time tomography, is motivated by the geophysical problem of recovering the inner structure of the earth (such as the sound speed or the index of refraction) from the travel times of seismic waves at the surface \cite{He1905, WZ1907}. Mathematically, the ordinary {\em boundary rigidity} problem (no presence of magnetic fields) deals with the reconstruction of a compact Riemannian manifold $(M,g)$ with smooth boundary from its boundary distance function $d_g$, where the value of $d_g:\p M\times \p M\to \mathbb R$ at any two boundary points $x,y$, denoted by $d_g(x,y)$, is defined as the infimum of the lengths of geodesics in $M$ connecting $x$ and $y$. Similarly the boundary distance function is invariant under the diffeomorphism (it is actually an isometry) mentioned above, thus one can only expect to reconstruct $g$ up to the natural obstruction.

There are examples showing that a general compact manifold with boundary may not be boundary rigid, see also Remark \ref{boundary rigid fail}, one needs to impose additional geometric conditions. One such condition is simplicity. A compact Riemannian manifold $M$ is {\em simple} if the boundary $\p M$ is strictly convex and any two points can be joined by a unique distance minimizing geodesic. It is known that on simple manifolds, the boundary rigidity problem and the lens rigidity problem are equivalent \cite{Mi81}. Michel \cite{Mi81} conjectured that simple manifolds are boundary rigid, Pestov and Uhlmann showed that this is true for simple surfaces \cite{PU05}. In higher dimensions, %simple manifolds satisfying some a priori curvature bounds or special symmetries are boundary rigid \cite{Gr83, BCG95}. %boundary rigidity holds on simple manifolds satisfying a priori curvature or symmetry conditions \cite{Gr83, BCG95}. 
Stefanov and Uhlmann proved that a generic simple metric is boundary rigid \cite{SU05}, they also gave stability estimates. See \cite{Cr04, SU08a, UZ16} for recent surveys on the ordinary boundary rigidity problem. Rigidity problems for simple magnetic systems were studied in \cite{DPSU07, Her12} (see the definition of simple magnetic systems there), \cite{DU10} provided a reconstruction procedure of simple 2D magnetic systems from scattering relations. Rigidity problems for more general Hamiltonian systems were considered in \cite{AZ15}.  

On non-simple manifolds, it is natural to consider the lens rigidity problem. Croke showed that the finite quotient of a lens rigid manifold is lens rigid \cite{Cr05}, and the torus is lens rigid too \cite{Cr14}. Lens rigidity also holds on real-analytic manifolds which satisfy some non-conjugacy condition \cite{Va09}. Stefanov and Uhlmann have shown lens rigidity for metrics close to a generic class of non-simple metrics including the real-analytic ones \cite{SU09}. Recently, Guillarmou \cite{Gui14} proved deformation lens rigidity for a class of manifolds with hyperbolic trapped sets, including the negatively curved manifolds. He also showed that on surfaces the lens data determines the metric up to conformal diffeomorphism under the same assumptions. However, generally manifolds with trapped geodesics are not necessarily lens rigid \cite{CK94}. Stability estimates for lens rigidity problem was studied in \cite{BZ14}. Reconstruction of a real-analytic magnetic system from its scattering relation was considered in \cite{HV11}.

Given a strictly magnetic convex boundary point $p$, the {\em partial data (or local) lens rigidity} problem for magnetic systems is whether we can determine the metric $g$ and magnetic field $\Omega$ near $p$ from the ({\em partial}) lens data known near $S_p\p M$. In this paper, we consider the case that $g$ is in some conformal class, i.e. $g=c^2g_0$ for some smooth positive function $c$, where $g_0$ is known, we want to recover the conformal factor $c$ (also the magnetic field $\Omega$). Rigidity for the full data problem in the same conformal class was proven in \cite{Mu81, MR78} for simple metrics, see also \cite{Cr91}. Then it was extended to simple magnetic systems by \cite{DPSU07}. Notice that there is no natural obstruction to the unique determination in the conformal case, one expects to recover the system uniquely from its lens data.

Let $\iota: \p M\rightarrow M$ be the canonical inclusion, $(L,\ell)$ and $(\tilde L,\tilde \ell)$ be the lens data of $(c^2g_0,\Omega)$ and $(\tilde c^2g_0,\tilde \Omega)$ respectively. Assume that the conformal factor and the tangential part of the magnetic field, i.e. $\iota^*\Omega$, are known on $\p M$ near some strictly magnetic convex point $p$, we get the following local magnetic lens rigidity result.

\begin{Theorem}\label{local lens rigid}
Let $n=$dim $M\geq 3$, let $c,\tc>0$ be smooth functions, $\Omega, \tilde \Omega$ be smooth closed 2-forms and let $\p M$ be strictly magnetic convex with respect to both $(c^2g_0,\Omega)$ and $(\tc^2g_0,\tilde\Omega)$ near $p\in\p M$. Assume that on $\p M$ near $p$, $c=\tilde c$ and $\iota^*\Omega=\iota^*\tilde{\Omega}$. If $L=\tL$, $\ell=\tl$ near $S_p\p M$, then $c=\tc$ and $\Omega=\tilde{\Omega}$ in $M$ near $p$.
%. Then $c=\tc$, $\Omega=\tO$ in $M$ near $p$.
\end{Theorem}

This is the first partial data rigidity result for smooth magnetic systems, and it generalizes local rigidity results from \cite{SUV16} of the geodesic flow. Previously such results are only known in the real-analytic category, see e.g. \cite{LSU03}. In the last section, we will discuss a weaker version of local lens rigidity problem for a general family of smooth curves.

\begin{Remark}
Notice that our data in Theorem \ref{local lens rigid} contains both the travel time data $\ell$ and the boundary restriction of $c$. Here $c$ is necessary for defining the lens data $L$ and $\ell$.  %One way to reduce the amount of initial data is to define the travel time data $\ell$ on $\p M\times \p M$, so $\ell(x,y)$ as the infimum of the lengths of magnetic geodesics from $x$ to $y$.  
\end{Remark}

Similar to \cite{SUV16}, the main idea of the proof of Theorem \ref{local lens rigid} is a reduction to a local uniqueness problem of the magnetic X-ray transform by applying an integral identity from \cite{SU98}. However, instead of considering the flow on the cotangent bundle as in \cite{SUV16}, we do all the analysis on the tangent bundle (see Section 2), which turns out to be much more convenient. Our approach even simplifies the original method of \cite{SUV16} for the ordinary boundary and lens rigidity problems. 

The geodesic X-ray transform (or {\em tensor tomography}) problem is concerned with the recovery of a function or tensor field from its integrals along geodesics joining boundary points, and it is the linearization of the boundary and lens rigidity problems. This problem has been extensively studied in the literature, see e.g. \cite{Mu77, PS88, Sh94, PU04, SU04, Da06, Sh07, SU08b, SU12, PSU13, PSU14, Mo14, MSU15, PZ15} and the references therein. Recent studies of the magnetic tensor tomography problem can be found in \cite{DPSU07, Ai12}. The {\em local} tensor tomography problem was considered in \cite{K09, KS09} for real-analytic metrics, \cite{UV15, SUV14} for smooth metrics; the case of magnetic geodesics appears in \cite[Appendix]{UV15} and \cite{ZZ16}. 

For the ordinary partial data rigidity problems, the reduction mentioned above ends with a local weighted geodesic ray transform of vector functions. It turns out that the local linear problem we need to consider in the current paper is the invertibility of some local magnetic ray transform of the combination of vector functions (related to the conformal factor) and 1-1 tensors (related to the Lorentz force, which can be viewed as a vector of $1$-forms) with matrix weights, see Section \ref{pseudo-linearization} and \ref{linear problem} for more details. Microlocal analysis of weighted X-ray transforms can be found in \cite{FSU08, HS10, Ho13} for the global problems and \cite{SUV16, PSUZ16} for the local problems. Comparing to the papers mentioned above, the matrix weights on the functions and 1-forms are different in our case, it is not an attenuated ray transform as in \cite{PSUZ16}. Generally the kernel of the X-ray transform of functions plus 1-forms is nontrivial (unless some `divergence free' condition is assumed), however, since the 1-forms in our X-ray transform satisfy an additional property, we can show that actually the kernel is trivial in our case. This is also one of the contributions of our paper.

We also study the magnetic boundary rigidity problem with partial data. The following boundary action functional on $\p M\times \p M$ was introduced in \cite{DPSU07} which plays the role of boundary distance function for simple magnetic systems
\begin{equation}\label{action functional}
\mathbb A(x,y):=T(x,y)-\int_{\gamma_{x,y}}\alpha,
\end{equation}
where $\gamma_{x,y}$ is the unique unit speed magnetic geodesic joining $x,y$ with $T(x,y)$ its length, and $\alpha$ is the magnetic potential (1-form) satisfying $\Omega=d\alpha$. The existence of $\alpha$ is due to the trivial topology of simple systems and the fact that $\Omega$ is closed. Notice that the definition of $\mathbb A$ depends on the choice of $\alpha$, given any $f\in C^{\infty}(M)$, $\Omega=d\alpha=d(\alpha+df)$. For the local problem, assume $p\in \p M$ is strictly magnetic convex, one investigates the determination of a magnetic system near $p$ from the knowledge of the boundary action functional near $(p,p)$. Notice that given a sufficiently small neighborhood of $p$ in $M$, we can always assume that it has trivial topology, thus such magnetic potentials always exist locally near $p$. Moreover, due to the strict convexity, points $x,y\in \p M$ near $p$ can be joined by a unique unit speed magnetic geodesic, therefore we can define a boundary action functional $\mathbb A$ locally near $(p,p)$. 

Note that in the geodesic case, the local boundary distance data is equivalent to the local lens data near a strictly convex point $p$. However, for magnetic systems, such equivalence only holds between the boundary action functional and the scattering relation \cite{DPSU07}. It is unclear that whether one can derive the travel time data (the length of $\gamma_{x,y}$) from $\mathbb A(x,y)$, except in the case of real-analytic systems \cite{HV11}. The following rigidity result related to the boundary action functional is an immediate corollary of Theorem \ref{local lens rigid}.

\begin{Theorem}\label{local boundary rigid}
Let $n=$dim $M\geq 3$, let $c,\tc>0$ be smooth functions, $\Omega, \tilde \Omega$ be smooth closed 2-forms and let $\p M$ be strictly magnetic convex with respect to both $(c^2g_0,\Omega)$ and $(\tc^2g_0,\tilde\Omega)$ near $p\in\p M$. Assume that there are 1-forms $\alpha$ and $\tilde\alpha$, satisfying $\Omega=d\alpha$ and $\tilde\Omega=d\tilde\alpha$, such that $\mathbb A=\tilde{\mathbb A}$, $T=\tilde T$ on $\p M\times \p M$ near $(p,p)$. Then $c=\tc$ and $\Omega=\tilde{\Omega}$ in $M$ near $p$.
\end{Theorem}

\begin{Remark}\label{ell and T}
One needs to know both $\mathbb A$ and $T$ to reduce Theorem \ref{local boundary rigid} to the settings of Theorem \ref{local lens rigid}. If $L(x,v)=(y,v')$, then  near $p$ we have $\ell(x,v)=T(x,y)$. The advantage of considering $T$, instead of $\ell$, is that we do not need to know the restriction of the conformal factor on the boundary. On the other hand, it is not difficult to see that one can determine $c$ on the boundary from $T$.
\end{Remark}

Another application of Theorem \ref{local lens rigid} is a global lens rigidity result under some foliation condition related to the magnetic flow as follows: Given a compact Riemannian manifold $(M,g)$ with smooth boundary and a magnetic field $\Omega$, we say that $M$ can be {\em foliated by strictly magnetic convex hypersurfaces} for the magnetic system $(M,g,\Omega)$ if there exist a smooth function $f:M\to \mathbb R$ and $a<b$, such  that the level set $f^{-1}(t)$ is strictly magnetic convex with respect to $(M,g,\Omega)$ from $f^{-1}((-\infty,t])$ for any $t\in (a,b]$, $df$ is non-zero on these level sets and $M\setminus f^{-1}((a,b])$ has empty interior. Note that $\p M$ is not necessarily a level set of $f$. The next theorem is an analog of \cite[Theorem 1.3]{SUV16}.

\begin{Theorem}\label{global lens rigid}
Let $n=$dim $M\geq 3$, let $c,\tc>0$ be smooth functions, $\Omega, \tilde \Omega$ be smooth closed 2-forms and let $\p M$ be strictly magnetic convex with respect to both $(c^2g_0,\Omega)$ and $(\tc^2g_0,\tilde\Omega)$. Assume that $c=\tilde c$ and $\iota^*\Omega=\iota^*\tilde{\Omega}$ on $\p M$, and $M$ can be foliated by strictly magnetic convex hypersurfaces for $(M,c^2g_0,\Omega)$. If $L=\tL$, $\ell=\tl$ on $\p_+SM$, then $c=\tc$ and $\Omega=\tilde{\Omega}$ in $M$.
\end{Theorem}

The global foliation condition implies non-trapping on $f^{-1}([t,b])$ for any $a<t<b$, but allows the existence of conjugate points on $M$. Notice that in \cite{Gui14} trapped sets are allowed, but the manifolds need to be free of conjugate points. In the case of absence of magnetic fields, the condition is an analog of the Herglotz and Wieckert \& Zoeppritz condition $\frac{\p}{\p r}\frac{r}{c(r)}>0$, see also \cite[Section 6]{SUV16}. Examples of manifolds satisfying the foliation conditions are compact submanifolds of complete manifolds with positive curvature \cite{GW76}, compact manifolds with non-negative sectional curvature \cite{Es86}, and compact manifolds with no focal points \cite{RS02}. Our foliation condition defined just before Theorem \ref{global lens rigid} is the corresponding version for magnetic systems.    

One can derive similar stability estimates for these rigidity problems as in \cite{SUV16}, however, in the current paper we only deal with uniqueness.

\begin{Remark}\label{boundary rigid fail}
Theorem \ref{global lens rigid} is regarding the lens rigidity problem. Generally under the foliation condition, the global boundary rigidity problem is not solvable. One example is the unit ball with radial symmetric metric and very large curvature at the center, then it is impossible to recover the metric near the center as no distance minimizing geodesics will reach a small neighborhood of the center.
\end{Remark}

The outline of the paper is as follows: In Section 2, we reduce the rigidity problems to a local X-ray transform problem through a pseudo-linearization. The inveribility of the X-ray transform is shown in Section 3 by using Melrose' scattering calculus. Section 4 consists the proofs of various local and global rigidity theorems. Finally, we discuss a weak version of the lens rigidity problem for general smooth curves in Section 5.

\medskip

\noindent{\bf Acknowledgments.} The author wants to thank Prof. Gunther Uhlmann for suggesting this problem to him. He is also grateful to both Prof. Gunther Uhlmann and Prof. Gabriel P. Paternain for reading an earlier version of the paper and very helpful discussions and comments. This research was supported by EPSRC grant EP/M023842/1.

%%%%%%%%%%%%%%%%%%%%%%%%%%%%%%%%%%%%%%%%%%%%%%%%%%%%%%%%%%%%%%%%%%%%%%%%%%%%
%----------------------------------------------------------------------------------------------------------------
\section{Reducing to a linear problem}\label{pseudo-linearization}

First we can determine the jets of $c$ and $\Omega$ at any boundary point $p$ at which $\p M$ is strictly magnetic convex from the scattering relation near $S_p\p M$.% and $\iota^*\Omega$ near $p$.

\begin{Lemma}\label{jet}
Let $c,\tc>0$ be smooth functions, $\Omega, \tilde \Omega$ be smooth closed 2-forms and let $\p M$ be strictly magnetic convex with respect to both $(c^2g_0,\Omega)$ and $(\tc^2g_0,\tilde\Omega)$ near a fixed $p\in\p M$. Let $c=\tilde c$, $\iota^*\Omega=\iota^*\tO$ on $\p M$ near $p$ and $L=\tL$ near $S_p\p M$. Then in any local coordinate system $\p^{a}c=\p^{a}\tc$, $\p^{a}\Omega=\p^{a}\tO$ on $\p M$ near $p$ for any multiindex $a$.
\end{Lemma}

If $\Omega=\Omega_{ij}dz^i\wedge dz^j$, we define $\p^{a}\Omega:=\{\p^{a}\Omega_{ij}\}$.

The proof of Lemma \ref{jet} was essentially given in \cite{DPSU07, HV11}, for the sake of completeness we give the sketch here.

\begin{proof}
By \cite[Lemma 2]{HV11}, there exist 1-forms $\alpha$ and $\tilde\alpha$ in a neighborhood $U$ of $p$ such that $\Omega=d\alpha$, $\tilde \Omega=d\tilde \alpha$ in $U$ and $\iota^*\alpha=\iota^*\tilde\alpha$ on $U\cap \p M$. Now since $\p M$ is strictly magnetic convex with respect to both magnetic systems near $p$, applying \cite[Lemma 2.6]{DPSU07}, the boundary action functional $\mathbb A$ and $\tilde{\mathbb A}$ (locally defined) with respect to $(c^2g_0,\alpha)$ and $(\tilde c^2g_0,\tilde \alpha)$ respectively are equal on $\p M\times \p M$ near $(p,p)$. Then \cite[Theorem 2.2]{DPSU07} implies that $\p^{a}c=\p^{a}\tc$, $\p^{a}\Omega=\p^{a}\tO$ on $\p M$ near $p$ for any multiindex $a$. Note that Lemma 2.6 and Theorem 2.2 of \cite{DPSU07} are global theorem on simple systems, but the proof works locally near a strictly magnetic convex boundary point if the local data is given. Moreover, for conformal metrics, the gauge in \cite[Theorem 2.2]{DPSU07} is trivial.
%\bigskip
%By strictly magnetic convexity assumption, any two points $z, z'$ on $\p M$ near $p$, there is a unique magnetic geodesic $\g$ in a small neighborhood of $p$ in $M$ connecting them, and all the interior of $\g$ belongs to the interior of $M$. Thus if $L(z,v)=(z',v')$, then locally we can wirte $\ell (z,v)$ as $\rho (z,z')$. Given $v\in T_z \p M$, let $\tau(s)$, $-\epsilon<s<\epsilon$ for $\epsilon$ sufficiently small, be a curve on $\p M$ with $\tau(0)=z$ and $\dot{\tau}(0)=v$. It is easy to check that 
%$$\lim_{s\to 0}\frac{\rho(z,\tau(s))}{s}=|v|_g.$$
%A similar equality holds for the system $(M,\tg,\tO)$. Since $\ell=\tl \Rightarrow \rho=\tilde{\rho}$, we have $|v|_g=|v|_{\tg}$ for all $v\in T_z\p M$, which means $\iota^*g=\iota^*\tg$. Now by following the arguments of \cite[Theorem 2]{HV11}, we get that in boundary normal coordinates, $\p^{\alpha}c=\p^{\alpha}\tc$, $\p^{\alpha}\Omega=\p^{\alpha}\tO$ on $\p M$ near $p$ for any multiindex $\alpha$.
\end{proof}

Similar to \cite{SUV16}, the starting point of the paper is an integral identity first proved in \cite{SU98} for the usual geodesic flow. Actually the proof of this identity works for general (not necessarily Hamiltonian) flows. Let $V,\, \tV$ be two vector fields on some smooth manifold $N$ (no metric assigned). Denote by $X(s,\X0)$ the solution of $\dot{X}=V(X),\, X(0)=\X0$, and we use the same notation for $\tV$ with the corresponding solution denoted by $\tX$. We state the identity here without proof, please see \cite[p87]{SU98} and \cite[Lemma 2.2]{SUV16} for the detailed proof.

\begin{Lemma}
For any $t>0$ and initial condition $\X0$, if $X(\cdot,\X0)$ and $\tX(\cdot,\X0)$ exist on the interval $[0,t]$, then 
\begin{equation}\label{identity1}
X(t,\X0)-\tX(t,\X0)=\int_0^t\frac{\p\tX}{\p\X0}(t-s,X(s,\X0))(V-\tV)(X(s,\X0))ds.
\end{equation}
\end{Lemma}

%For magnetic flows, since we focus our interests on magnetic geodesics with unit speed. We can consider initially the Lorentz force $Y$ is only defined on $SM$, then we extend $Y$ onto $TM$ by defining 
%$$Y_z(v)=|v|_gY_i^j(z)v^i\frac{\p}{\p z^j},$$
%we denote the extended bundle map still by $Y$, thus $Y$ is homogeneous of degree 2. 
One can check that given the metric $g$, the restriction of the Lorentz force $Y$ on $SM$ uniquely determines the magnetic field $\Omega$, vice versa. %Moreover, by investigating the motion equation \eqref{Mag},
%$$\nabla_{c\dot{\g}}c\dot{\g}=c^2\nabla_{\dot{\g}}\dot{\g}=c^2Y(\dot{\g})=Y(c\dot{\g}), \, c>0$$
%i.e. the magnetic flows of different energy levels will induce the same magnetic geodesics but with different constant speeds on $M$. We do the same thing for $\tY$. 
From now on, we use $(g, Y)$ and $(\tg, \tY)$ to represent the magnetic systems. Note that given any $u\in TM$, $\<Y(u),u\>_g=\Omega(u,u)=0$.
%\emph{Remark:} For the newly defined $Y$ and $\tY$, $\Omega=\tilde{\Omega}$ implies $c \,Y=\tc\, \tY$. 

For $(g, Y)$, let $(z,v)\in SM$ (with respect to metric $c^2g_0$), then by \eqref{generator}
\begin{equation}\label{V}
V(z,v)=\big(\frac{dz}{dt}, \frac{dv}{dt}\big)|_{t=0}=\big(v, -\Gamma^i_{jk}(z)v^jv^k\frac{\p}{\p v^i}+Y^i_j(z)v^j\frac{\p}{\p v^i}\big).
\end{equation}
On the other hand, if we let $(z,v)$ be the initial vector for some magnetic flow $(\tga, \dot{\tga})$ of $(\tg, \tY)$ with fixed energy level, then generally $(z,v)$ is not on the unit sphere bundle with respect to metric $\tg$. %, that's the reason we extend the definition of $\tY$ in above way to make it have homogeneous degree 2. Thus the corresponding magnetic geodesic will be the same as the one w.r.t. $(z, v/|v|_{\tg})$. 
Since $|v|_g=1$, we get that $\tga$ has constant speed $|v|_{\tg}=\sqrt{\tc^2|v|^2_{g_0}}=\sqrt{\tc^2/c^2}=\tc/c$. Then similar to $V$, we have  
\begin{equation}\label{tilde V}
\tV(z,v)=\big(\frac{d\tilde z}{d\tilde{t}},\frac{d\tilde v}{d\tilde{t}}\big)|_{\tilde t=0}=\big(v, -\tilde{\Gamma}^i_{jk}(z)v^jv^k\frac{\p}{\p v^i}+\tY^i_j(z)v^j\frac{\p}{\p v^i}\big).
\end{equation}

\begin{Remark}
If we consider the Hamiltonian flow on the cotangent bundle with the Hamiltonian $H(\xi)=\frac{1}{2}|\xi|_g^2=\frac{1}{2}g^{ij}\xi_i\xi_j$, then the corresponding generating vector $V$, $dH=\beta\,(V\, ,\,\cdot)$, is
\begin{equation*}
\begin{split}
V & =g^{ij}\xi_j\frac{\p}{\p z^i}+\Big(-\frac{1}{2}\frac{\p g^{ij}}{\p z^k}\xi_i\xi_j+(\Omega_{ik}-\Omega_{ki})g^{ij}\xi_j\Big)\frac{\p}{\p\xi_k}\\
& =g^{ij}\xi_j\frac{\p}{\p z^i}+\Big(-\frac{1}{2}\frac{\p g^{ij}}{\p z^k}\xi_i\xi_j+Y^j_k\xi_j\Big)\frac{\p}{\p\xi_k}.
\end{split}
\end{equation*}
However, for the cotangent bundle, the proof will be more complicated, since both the $\p_z$ and $\p_\xi$ terms of $V$ and $\tilde V$ are different, see \cite{SUV16}. The arguments in the current paper also simplifies the proof of \cite{SUV16} for the geodesic case.
\end{Remark}

We denote points on the tangent bundle $TM$ by $\s=(z,v)$, then the flow with initial point $\s$ is $X(t,\s)=(Z(t,\s), \Xi(t,\s))$. Thus the identity \eqref{identity1} is rewritten as
\begin{equation}\label{identity2}
X(t,\s)-\tX(t,\s)=\int_0^t\frac{\p\tX}{\p\s}(t-s,X(s,\s))(V-\tV)(X(s,\s))ds.
\end{equation}
Appling the lens data, note that $c=\tilde c$ on $\p M$ near $p$ by Lemma \ref{jet}, we get the following proposition:

\begin{Proposition}\label{identity from lens data}
Assume $L(z_0,v_0)=\tL(z_0,v_0),\, \ell(z_0,v_0)=\tl(z_0,v_0)$ for some $\s_0=(z_0,v_0)\in \p_+SM$. Then
\begin{equation}\label{identity3}
\int_0^{\ell(\s_0)}\frac{\p\tX}{\p\s}(\ell(\s_0)-s,X(s,\s_0))(V-\tV)(X(s,\s_0))ds=0.
\end{equation}
\end{Proposition}

We take the second $n$-dimensional components of \eqref{identity3} and plug in \eqref{V} and \eqref{tilde V} to the identity. We assume that both $X(t,\sigma)$ and $\tilde X(t,\sigma)$ exist for $t\in [0,\ell(\sigma)]$. We get
\begin{equation}\label{identity4}
\int_0^{\ell(\s)}\frac{\p\tilde{\Xi^l}}{\p v^i}(\ell(\s)-s,X(s,\s))((\tilde{\Gamma}^i_{jk}-\Gamma^i_{jk})v^jv^k+(Y^i_j-\tY^i_j)v^j)(X(s,\s))ds=0.
\end{equation}  
for $l=1,2,\cdots, n$ and any $\s\in \p_+SM$ if the two systems $(g,Y)$ and $(\tg,\tY)$ have the same lens data at $\sigma$. Actually we can integrate over $s\in\R$ after extending the magnetic geodesics. To see this, by Lemma \ref{jet} any smooth extension of $(g,Y)$ near $p\in \p M$ is also a smooth extension of $(\tg,\tY)$, we denote the extended manifold by $\widetilde M$, thus we can assume outside $M$ near $p$, $(g,Y)=(\tg,\tY)$, which means $V=\tV$ outside $M$ near $p$. So the integrand of \eqref{identity4} vanishes for $s\notin[0,\ell(\s)]$.

Now we do more analysis of the term $V-\tV$.
\begin{equation*}
\begin{split}
\tilde{\Gamma}^i_{jk}-\Gamma^i_{jk} & =\frac{1}{2}\tg^{il}(\frac{\p\tg_{lk}}{\p z^j}+\frac{\p\tg_{lj}}{\p z^k}-\frac{\p\tg_{jk}}{\p z^l})-\frac{1}{2}g^{il}(\frac{\p g_{lk}}{\p z^j}+\frac{\p g_{lj}}{\p z^k}-\frac{\p g_{jk}}{\p z^l}) \\
& =\frac{1}{\tc}(\delta^i_k\frac{\p\tc}{\p z^j}+\delta^i_j\frac{\p\tc}{\p z^k}-g_0^{il}{g_0}_{jk}\frac{\p\tc}{\p z^l})-\frac{1}{c}(\delta^i_k\frac{\p c}{\p z^j}+\delta^i_j\frac{\p c}{\p z^k}-g_0^{il}{g_0}_{jk}\frac{\p c}{\p z^l}) \\
& =\delta^i_k\frac{\p \ln(\tc/c)}{\p z^j}+\delta^i_j\frac{\p \ln(\tc/c)}{\p z^k}-g_0^{il}{g_0}_{jk}\frac{\p \ln(\tc/c)}{\p z^l}.
\end{split}
\end{equation*}
We denote $\ln\frac{\tc}{c}=h$, then
$$(\tilde{\Gamma}^i_{jk}-\Gamma^i_{jk})v^jv^k=2\frac{\p h}{\p z^j}v^jv^i-g_0^{ij}\frac{1}{c^2}\frac{\p h}{\p z^j}.$$
Thus \eqref{identity4} can be rewritten as 
\begin{equation}\label{identity5}
\int\frac{\p\tilde{\Xi^l}}{\p v^i}(\ell(\s)-s,X(s,\s))\Big((2v^iv^j-g_0^{ij}\frac{1}{c^2})\frac{\p h}{\p z^j}+(Y^i_j-\tY^i_j)v^j\Big)(X(s,\s))ds=0.
\end{equation}

Let exit times $\tau(z,v)$ be the minimal (and the only) $t\geq 0$ such that $X\big(t,(z,v)\big)\in\p_-SM$. They are well defined near $S_p\p M$ if $\p M$ is strictly magnetic convex at $p$. We write $(z,v)=X(s,\s)$, then 
$$\frac{\p\tilde{\Xi^l}}{\p v^i}(\ell(\s)-s,X(s,\s))=\frac{\p\tilde{\Xi^l}}{\p v^i}(\tau(z,v), (z,v))=\frac{\p\tilde{\Xi^l}}{\p v^i}(z,v).$$
As we have extended the system outside $M$ near $p$, $\tau$ is smooth near $S_p\p M$ on $S\widetilde M$. Thus we get 
\begin{equation}\label{identity6}
I_{AB}[\varphi,\Phi](\gamma)=\int_{\mathbb R} A(\gamma(t),\dot{\gamma}(t))\Big(B(\gamma(t),\dot{\gamma}(t))\,\varphi(\gamma(t))+\Phi(\dot{\gamma}(t))\Big)\,dt=0,
\end{equation}
where $A(z,v)=(\frac{\p\tilde{\Xi}^l}{\p v^i}(z,v))_{n\times n}$ and $B(z,v)=(2v^iv^j-g_0^{ij}\frac{1}{c^2})_{n\times n}$ are two smooth matrix weights, $\varphi=(\frac{\p h}{\p z^1},\cdots,\frac{\p h}{\p z^n})$ is a vector-valued function and $\Phi=(Y^i_j-\tilde Y^i_j)$ is an $(1,1)$-tensor (can be viewed as a vector of 1-forms). Hence $I_{AB}$ is a matrix weighted magnetic ray transform of the combination of vector functions and $1$-forms. Notice that by Lemma \ref{jet}, $\varphi$ and $\Phi$ vanish outside $M$, the integral is actually on a finite interval.

In particular, since $\frac{\p\tilde{\Xi}^l}{\p v^i}(z,v)=\delta^l_i$ on $S_p\p M$, i.e. $A(z,v)=Id_{n\times n}$ for $(z,v)\in S_p\p M$, we get that the matrix $A$ is invertible near $S_p\p M$ in $S\widetilde M$ (with respect to $c^2g_0$). On the other hand, the weight $B$ is also an invertible matrix. To see this, notice that $B=2v^Tv-g^{-1}$ where $v=(v^1,\cdots, v^n)$ and $g^{-1}$ is the dual of the metric $g=c^2g_0$. So $gB=2gv^Tv-Id$, and the invertibility of $B$ is equivalent to the invertibility of $gB$. To show that $gB$ is invertible, let $u$ be an arbitrary column vector such that $gBu=0$, then $vgBu=0$, i.e. $2vgv^Tvu-vu=0$. Note that the vector $v\in T_z\widetilde M$ has unit length, i.e. $vgv^T=1$, we get that $2vu=vu$, so $vu=0$. Since $gBu=2gv^Tvu-u=u$, this implies that $u=0$. Thus the matrix function $B$ is invertible on $S\widetilde M$ (with respect to $g$).

Now we have reduced the lens rigidity (non-linear) problem to the following magnetic ray transform (linear) problem:

{\it If $I_{AB}[\varphi,\Phi](\gamma)=0$, $[\varphi,\Phi]$ supported in $M$, for magnetic geodesics close to the ones tangent to $\p M$ at $p$, does that imply $[\varphi,\Phi]=0$ near $p$?}

We will show that the answer to above question is affirmative in the next section, actually it holds for weights $A,\, B$ more general than the specific ones of \eqref{identity5}.

%%%%%%%%%%%%%%%%%%%%%%%%%%%%%%%%%%%%%%%%%%%%%%%%%%%%%%%%%%%%%%%%%%%%%%%%%%%%%%%%%%

\section{Injectivity of the weighted ray transform}\label{linear problem}

\subsection{A scattering $\Psi DO$ and its kernel}
Now let $\rho\in C^{\infty}(\tM)$ be a boundary defining function of $\p M$, so that $\rho\geq 0$ on $M$. Suppose $\p M$ is strictly magnetic convex at $p\in \p M$, then given a magnetic geodesic $\g$ on $\widetilde M$ with $\g(0)=p$, $\dot{\g}(0)\in S_p\p M$, one has 
$$\frac{d^2\rho}{dt^2}(\g(t))|_{t=0}=-\Lambda(p,\dot{\g}(0))+\<Y_p(\dot{\g}(0)), \nu(p)\><0.$$
Similar to \cite{UV15} we consider the function $\tilde{x}(z)=-\rho(z)-\e |z-p|^2$ for some small enough $\e>0$, which makes $\tilde x$ strictly magnetic concave from $U_c=\{\tilde{x}\geq -c\}\subset \widetilde M$ for some sufficiently small $c>0$. Here $|\cdot|$ in the definition of $\tilde x$ is the Euclidean norm, since we are considering (local) rigidity problems near $p$, it is well-defined under some initially chosen local chart near $p$. Let $x=\tilde{x}+c$, the open set we will work with is 
$$O_c=U_c^{int}\cap \overline M=\{x>0, \rho\geq 0\}$$
with compact closure. For the sake of simplicity, we drop the subscript $c$, i.e. $U_c=U,\,O_c=O$.

One can complete $x$ to a coordinate system $(x,y)$ on a neighborhood of $p$, such that locally the metric is of product type $g=dx^2+g_y$ where $g_y$ is a metric on the level sets of $x$.  For each point $(x,y)$ we can parameterize magnetic geodesics through this point which are `almost tangent' to level sets of $x$ (these are the curves that we are interested in) by $\lambda\p_x+\omega\p_y\in TU$, $\omega\in\mathbb S^{n-2}$ and $\lambda$ is relatively small. Given a magnetic geodesic $\gamma$ with $\gamma(0)=(x,y)$, which is parameterized by $\lambda\p_x+\omega\p_y$ (i.e. $\gamma_{x,y,\lambda,\omega}(t)=(x'(t),y'(t))$), define $\alpha(x,y,\lambda,\omega):=\frac{d^2}{dt^2}x'(0)$. In particular, $\alpha(x,y,0,\omega)>0$ for $x$ small by the concavity of $x$. Furthermore, it was shown in \cite{UV15} that there exist $\delta_0>0$ small and $C>0$ such that for $|\l|\leq C\sqrt{x}$ (and $|\l|<\delta_0$), $x'(t)\geq 0$ for $t\in (-\delta_0,\delta_0)$, the magnetic geodesics remain in $\{x\geq 0\}$ at least for $|t|<\delta_0$, i.e. they are $O$-local magnetic geodesics (exit $O$ from $\p M$) for sufficiently small $c$. Note that \cite{UV15} considers geodesics, but the settings work for general curves, see the appendix of \cite{UV15}. 

\begin{equation*}
\begin{split}
I_{AB}[\varphi,\Phi](\xy) & =\int_{\R} \Big(A(B\varphi+\Phi)\Big)(\gxy(t),\dot{\gamma}_{x,y,\lambda,\omega}(t))\, dt\\
& =\int_{\mathbb R}\Big(A(B\varphi+\Phi)\Big)(x'(t),y'(t),\lambda'(t),\omega'(t))\,dt.
\end{split}
\end{equation*}
Following the approach of \cite{UV15, SUV14}, let $\chi$ be a smooth non-negative even function on $\mathbb R$ with compact support, and $W=\begin{pmatrix}x^{-1}Id_{n\times n} & 0\\0 & Id_{n\times n}\end{pmatrix}$, we consider the following operator, for $F>0$,
\begin{equation*}
\begin{split}
& N_{AB}[\varphi,\Phi](x,y)\\
=& W^{-1}e^{-F/x}\int\int\begin{pmatrix} x^{-2}B^* \\ g_{sc}(\lambda\p_x+\omega\p_y) \end{pmatrix}A^*\chi(\lambda/x)\Big(I_{AB}e^{F/x}W\Big)\begin{pmatrix}\varphi \\ \Phi \end{pmatrix}(x,y,\lambda,\omega)\,d\lambda d\omega\\
= & \begin{pmatrix} N_{00} & N_{01}\\ N_{10} & N_{11}\end{pmatrix}\begin{pmatrix} \varphi \\ \Phi\end{pmatrix}(x,y)
\end{split}
\end{equation*}
with
\begin{equation*}
\begin{split}
N_{00}\varphi &=xe^{-F/x}\int\int x^{-2}\chi(\lambda/x)B^*A^*\Big(\int ABe^{F/x}x^{-1}\varphi\,dt\Big) d\lambda d\omega,\\
N_{01}\Phi & =xe^{-F/x}\int\int x^{-2}\chi(\lambda/x)B^*A^*\Big(\int Ae^{F/x}\Phi\,dt\Big) d\lambda d\omega,\\
N_{10}\varphi & =e^{-F/x}\int\int \chi(\lambda/x)g_{sc}(\lambda\p_x+\omega\p_y)A^*\Big(\int ABe^{F/x}x^{-1}\varphi\,dt\Big) d\lambda d\omega,\\
N_{11}\Phi & =e^{-F/x}\int\int \chi(\lambda/x)g_{sc}(\lambda\p_x+\omega\p_y)A^*\Big(\int Ae^{F/x}\Phi\,dt\Big) d\lambda d\omega,\\
\end{split}
\end{equation*}
where $g_{sc}$ is a {\it scattering metric}, locally it can be written as $g_{sc}=x^{-4}dx^2+x^{-2}h(x,y)$, here $h$ is the metric on the level sets of $x$. 
%Obviously $A_F\in \mbox{hom}(T^*_{sc}U\times \underline U ,\,T^*_{sc}U\times \underline U)$ and $B_F\in \mbox{hom}(Sym^2T^*_{sc}U\times T^*_{sc}U,\,Sym^2T^*_{sc}U\times T^*_{sc}U)$, where $\underline U$ is the trivial bundle. 
The local basis for the scattering cotangent bundle $^{sc}T^*U$ is $\frac{dx}{x^2},\,\frac{dy_1}{x}\cdots \frac{dy_{n-1}}{x}$, and its dual $^{sc}TU$, scattering tangent bundle, has the local basis $x^2\p x,\,x\p y_1\cdots x\p y_{n-1}$.

Comparing with the exponentially conjugated operators introduced in \cite{UV15, SUV14}, we add an additional conjugacy $W$. This additional conjugacy helps to make $N_{ij}$, $i,j=1,2$ of the same order, see Lemma \ref{Schwartz kernel}. Similar idea appears in \cite{PSUZ16, ZZ16}. From now on, we assume that both $A$ and $B$ are smooth invertible matrix-valued function on $SU$, and $B$ is even on the momentum variable, i.e. $B(z,v)=B(z,-v)$. We will show that $N_{AB}$ is an elliptic scattering pseudodifferential opeartor (see \cite{Mel94, UV15, SUV14} for more details) for such choice of $A,\, B$.

It is well known, see e.g. \cite{DPSU07, FSU08}, that the maps
$$\Gamma_+: S\widetilde M\times [0,\infty)\rightarrow [\widetilde M\times\widetilde M; \diag],\, \Gamma_+(z,v,t)=(z, |z'-z|, \frac{z'-z}{|z'-z|})$$
and
$$\Gamma_-: S\widetilde M\times (-\infty,0]\rightarrow [\widetilde M\times\widetilde M; \diag],\, \Gamma_-(z,v,t)=(z, -|z'-z|, -\frac{z'-z}{|z'-z|})$$
are two diffeomorphisms near $S\widetilde M\times \{0\}$. Here $[\widetilde M\times\widetilde M; \diag]$ is the {\it blow-up} of $\widetilde M$ at the diagonal $z=z'$.

Similar to \cite{UV15}, we can also use $(x,y,|y'-y|,\frac{x'-x}{|y'-y|},\frac{y'-y}{|y'-y|})$ as the local coordinates on $\Gamma_+(\supp\tilde\chi\times[0,\delta_0))$ for small $\delta_0$; and analogously for $\Gamma_-(\supp\tilde\chi\times (-\delta_0,0])$ the coordinates are $(x,y,-|y'-y|,-\frac{x'-x}{|y'-y|},-\frac{y'-y}{|y'-y|})$. Indeed, this corresponds to the fact that we are using $(x,y,\l,\o)$ with $\o\in \S^{n-2}$, instead of $S\tM$, to parameterize curves, when $|y'-y|$ is large relative to $x'-x$, i.e. in our region of interest.

As we want to study the scattering behavior of the operators $N_{AB}$ up to the front face $x=0$, we instead apply the scattering coordinates $(x,y,X,Y)$, where 
$$X=\frac{x'-x}{x^2}, \, Y=\frac{y'-y}{x}.$$
Under the scattering coordinates
$$dt\, d\lambda\, d\omega=x^2|Y|^{1-n}J(x,y,X,Y)\,dXdY$$
with $J|_{x=0}=1$.
Note that on the blow-up of the scattering diagonal, $\{X=0,Y=0\}$, in the region $|Y|>\epsilon |X|$, thus on the support of $\chi$
\begin{equation}\label{scattering coordinate}
(x,y,|Y|,\frac{X}{|Y|},\hat Y)\quad \mbox{and} \quad (x,y,-|Y|,-\frac{X}{|Y|},-\hat Y)
\end{equation}
are valid coordinates, $\hat Y=\frac{Y}{|Y|}$, with $\pm |Y|$ being the defining functions of the front face of this blow up. 

It was shown in \cite{SUV14} that under the coordinates \eqref{scattering coordinate} and the scattering tangent and cotangent bases
\begin{equation}\label{lambda omega}
\begin{split}
& g_{sc}\Big((\lambda\circ\Gamma_{\pm}^{-1})\p_x+(\omega\circ\Gamma_{\pm}^{-1})\p_y\Big)\\
= & x^{-1}\bigg(\Big(\frac{X-\alpha(x,y,x|Y|,\frac{xX}{|Y|},\hat Y)|Y|^2}{|Y|}+x\tilde\Lambda_\pm(x,y,x|Y|,\frac{x
  X}{|Y|},\hat Y)\Big)\,\frac{dx}{x^2}\\
&\qquad\qquad+\Big(\hat Y+x|Y|\tilde\Omega_\pm(x,y,x|Y|,\frac{x
  X}{|Y|},\hat Y)\Big)\,\frac{h(\p_y)}{x}\bigg)
\end{split}
\end{equation}
and
\begin{equation}\label{lambda omega prime}
\begin{split}
& (\lambda'\circ\Gamma_{\pm}^{-1})\p_x+(\omega'\circ\Gamma_{\pm}^{-1})\p_y\\
= & x^{-1}\bigg(\Big(\frac{X+\alpha(x,y,x|Y|,\frac{xX}{|Y|},\hat
  Y)|Y|^2}{|Y|}+x|Y|^2\tilde\Lambda'_\pm(x,y,x|Y|,\frac{x
  X}{|Y|},\hat Y)\Big)\,x^2\p_x\\
&\qquad\qquad+\Big(\hat Y+x|Y|\tilde\Omega'_\pm(x,y,x|Y|,\frac{x
  X}{|Y|},\hat Y)\Big)\,x\p_y\bigg).
\end{split}
\end{equation}
From now on we denote $\frac{X-\alpha(x,y,x|Y|,\frac{xX}{|Y|},\hat Y)|Y|^2}{|Y|}$ by $S$, so $\frac{X+\alpha(x,y,x|Y|,\frac{xX}{|Y|},\hat Y)|Y|^2}{|Y|}=S+2\alpha|Y|$. It is not difficult to see that the Schwartz kernel of $N_{AB}$ has the following form

\begin{Lemma}\label{Schwartz kernel}
The Schwartz kernel of $N_{AB}$ at the scattering front face $x=0$ is
\begin{equation*}
K_{AB}(0,y,X,Y)=e^{-FX}|Y|^{1-n}\chi(S)\begin{pmatrix} K_{00} & K_{01} \\ K_{10} & K_{11}\end{pmatrix},
\end{equation*}
where
\begin{equation*}
\begin{split}
K_{00} & =B^*A^*AB,\\
K_{01} & =B^*A^*A\Big((S+2\alpha |Y|)(x^2\p_x)+\hat Y (x\p_y)\Big),\\
K_{10} & =\Big(S\frac{dx}{x^2}+\hat Y \frac{dy}{x}\Big)A^*AB,\\
K_{11} & =\Big(S\frac{dx}{x^2}+\hat Y \frac{dy}{x}\Big)A^*A\Big((S+2\alpha |Y|)(x^2\p_x)+\hat Y (x\p_y)\Big).
\end{split}
\end{equation*}
In particular, $N_{AB}$ is a scattering pseudodifferential operator of order $(-1,0)$ on $U$, i.e. $N_{AB}\in \Psi^{-1,0}_{sc}(U)$.
\end{Lemma}

Generally the Schwartz kernel of a scattering pseudodifferential operator has the form $x^\ell K$ with non-zero $K$ smooth in $(x,y)$ down to $x=0$. For our case, the zero in the superscript of $\Psi^{-1,0}_{sc}$ means exactly that $\ell=0$, while the number $-1$, related to $K$, has the meaning similar to the order of standard pseudodifferential operators.

\subsection{Ellipticity of $N_{AB}$}

To show that $I_{AB}$ is invertible (locally near $p\in \p M$) on the space $\{[\varphi,\Phi]: \varphi, \Phi\,\,\mbox{supported in}\,\, M,\<\Phi(u),u\>_g=0,\, \forall u\in TU\}$ (notice that in the non-linear problem, $\Phi=Y-\tilde Y$, since $g$ and $\tilde g$ are conformal, we get $\<\Phi(u),u\>_g=0$ for any $u$), we need to study the ellipticity of $N_{AB}$ which consists the main part of this subsection.

First we do some more analysis of the structure of $\Phi$ under special coordinates. Notice that to analyze the principal symbol of $N_{AB}$ at some point $z=(x,y)$, we can assume that the basis $\{\p_x,\p_y\}$ is orthonormal at $z$, i.e. the metric $g(z)=dx^2+dy^2$. Under such basis, by the property $\<\Phi(u),u\>_g=0$, $\Phi$ can be written as an antisymmetric matrix at $z$. In particular the diagonal of $\Phi$, $\diag \Phi$, is zero. Now if we rewrite $\Phi$ under the corresponding scattering metric $g_{sc}(z)=\frac{dx^2}{x^4}+\frac{dy^2}{x^2}$ (i.e. the dual basis $\{\frac{dx}{x^2},\frac{dy}{x}\}$), note that we consider $\Phi$ as a column vector of $n$ 1-forms, it has the following form, with $\Phi^i_j$ the element on the $i$-th row, $j$-th column,
\begin{equation}\label{Phi under scattering metric}
^{sc}\Phi=(\Phi^i_j)=\begin{pmatrix} 0 & \Theta \\ -x\Theta^T & \widetilde\Phi \end{pmatrix},
\end{equation}
where $\Theta$ is a row vector of $n-1$ components with transpose $\Theta^T$, $\widetilde\Phi$ is again an antisymmetric matrix but with order $(n-1)\times(n-1)$. In particular $\diag \,^{sc}\Phi$ is also zero under the scattering basis. 

Above discussion is under special coordinates, in general if we consider $\Phi$ as a $(1,1)$-tensor under the basis $\{\p_x\otimes \frac{dx}{x^2}, \p_x\otimes \frac{dy}{x}, \p_y\otimes \frac{dx}{x^2}, \p_y\otimes \frac{dy}{x}\}$, or equivalently as a bundle map $\Phi:\,^{sc}TU\to TU$, it satisfies the following equation
\begin{equation}\label{scattering magnetic property}
\<\Phi(Pu),u\>_g=0
\end{equation}
for any $u\in TU$, where $Pu:=P(u^x\p_x+u^y\p_y)=\frac{u^x}{x^2}(x^2\p_x)+\frac{u^y}{x}(x\p_y)$. Note that \eqref{scattering magnetic property} works for any $g=dx^2+g_y$.

We are interested in the asymptotic behavior of the principal symbol of $N_{AB}$ acting on $(z,\zeta)\in\, ^{sc}T^*U$. Since $U$ is a manifold with boundary, there are two types of asymptotic behaviors. First we consider the symbol at the {\em fiber infinity} of $^{sc}T^*U$, i.e. as $|\zeta|\to \infty$, this corresponds to the {\em standard} principal symbol as usual. Since the Schwartz kernel of $N_{AB}$ is smooth in $(x,y)$ down to $x=0$, it suffices to investigate the principal symbol at fiber infinity of $^{sc}T^*_{\p U}U$ ($\p U=\{x=0\}$).

\begin{Lemma}\label{fiber infinity}
For any $F>0$, $N_{AB}$ is elliptic, acting on the space $\{[\varphi,\Phi]\in C^{\infty}(U)^n\times\, ^{sc}T^*U^n : \<\Phi(Pu),u\>_g=0, \forall u\in TU\}$, at fiber infinity of the scattering cotangent bundle $^{sc}T^*U$.
\end{Lemma}

\begin{proof}
The analysis of the principal symbol of $N_{AB}$ at fiber infinity is quite similar to the standard microlocal analysis, i.e. the analysis of the conormal singularity of the standard principal symbol of $N_{AB}$ at the diagonal, $X=Y=0$, see e.g. \cite[Lemma 3.4]{SUV14}. Following the discussion above, we only need to study the behavior at the scattering front face $\{x=0\}$.

Changing the coordinates $(X,Y)\to (|Y|,\tilde S,\hat Y)$ with $\tilde S=X/|Y|$, in view of the compact supported $\chi$, the leading order behavior of the Fourier transform of $K_{AB}$ as $|\zeta|\to \infty$ is encoded in the integration of it, after removing the singular factor $|Y|^{-n+1}$, along the orthogonal equatorial sphere corresponding to $\zeta$, i.e. those $(\tilde S,\hat Y)$ with $\xi\tilde S+\eta\cdot\hat Y=0$. Notice that in this case, the extra vanishing factor $|Y|$ in $S=\tilde S-\alpha|Y|$ can be dropped, thus the principal symbol of $N_{AB}$ is essentially of the following form
\begin{equation*}
|\zeta|^{-1}\int_{\zeta^{\perp}\cap(\mathbb R\times \mathbb S^{n-2})}\chi(\tilde S)\begin{pmatrix} B^* \\ \tilde S \\ \hat Y \end{pmatrix}A^*A\begin{pmatrix}B & \tilde S & \hat Y\end{pmatrix}\,d\tilde S d\hat Y.
\end{equation*} 

Given any non-zero pair $[\varphi, \Phi]$, $\Phi=(\Phi^0,\Phi')$ with $\Phi^0$ and $\Phi'$ corresponding to the coefficients for the covectors $\frac{dx}{x^2}$ and $\frac{dy}{x}$ respectively,
$$(\sigma_p(N_{AB})[\varphi,\Phi],[\varphi,\Phi])=|\zeta|^{-1}\int_{\zeta^{\perp}\cap(\mathbb R\times \mathbb S^{n-2})}\chi(\tilde S)\,\Big|A\Big(B(0,y,\tilde S,\hat Y)\varphi+\tilde S \Phi^0+\hat Y\cdot\Phi'\Big)\Big|^2\,d\tilde S d\hat Y.$$
Now to prove the ellipticity of $N_{AB}$, it suffices to show that there is $(\tilde S,\hat Y)\in \zeta^{\perp}\cap(\mathbb R\times \mathbb S^{n-2})$ such that $\chi(\tilde S)>0$ and $B\varphi+\tilde S\Phi^0+\hat Y\cdot\Phi'\neq 0$ (notice that $A$ is invertible). We prove by contradiction, assume that for any $(\tilde S,\hat Y)\in \zeta^{\perp}\cap(\mathbb R\times \mathbb S^{n-2})$ with $\chi(\tilde S)>0$, $B\varphi+\tilde S\Phi^0+\hat Y\cdot \Phi'=0$. Notice that $\chi$ is even, if $\chi(\tilde S)>0$, then $\chi(-\tilde S)>0$, thus $B\varphi-\tilde S\Phi^0-\hat Y\cdot \Phi'=0$ (since $B(z,v)=B(z,-v)$), which implies that $\tilde S\Phi^0+\hat Y\cdot \Phi'=0$ and $B\varphi=0$. As $B$ is invertible too, we get $\varphi=0$.

On the other hand, considering the $n\times n$ matrix $\Phi$ as in \eqref{Phi under scattering metric}, we can find $n-1$ linearly independent elements from the set $\{(\tilde S,\hat Y): \xi\tilde S+\eta\cdot \hat Y=0,\, \chi(\tilde S)>0\}$ (here we need the dimension $n$ be at least $3$, if $n=2$ the set might be empty) such that $\tilde S\Phi^0+\hat Y\cdot \Phi'=0$, by the linear algebra, this implies that $\Phi^i=(\Phi^i_1,\cdots,\Phi^i_n)$ is parallel to $\zeta$ for all $i=1,2,\cdots,n$, i.e. there exists $c=(c^1,\cdots,c^n)$ such that $\Phi^i=c^i\zeta$. Then since $\diag \Phi=0$ under the orthonormal basis at $z=(0,y)$, i.e. $\Phi^i_i=0$ for $i=1,2,\cdots,n$, and all $\{\Phi^i\}_{i=1}^n$ are parallel, we derive that $\Phi^i_j=0,\,\forall \,1\leq i,j\leq n$, i.e. $\Phi=0$ which is a contradiction as the pair $[\varphi,\Phi]$ is non-zero. Thus the lemma is proved.
\end{proof}

Another type of asymptotic behaviors is at {\em finite points} of $^{sc}T^*_{\p U}U$, i.e. at $(z,\zeta)\in \,^{sc}T^*U$ with $z\in \p U$, which defines the {\em boundary (or scattering) principal symbol}. Again, such behavior is uniform for $z$ sufficiently close to $\p U$ due to the smoothness of the Schwartz kernel on $z$.

\begin{Lemma}\label{finite points}
For any $F>0$, $N_{AB}$ is elliptic, acting on the space $\{[\varphi,\Phi]\in C^{\infty}(U)^n\times\, ^{sc}T^*U^n : \<\Phi(Pu),u\>_g=0, \forall u\in TU\}$, at finite points of the scattering cotangent bundle $^{sc}T^*U$.
\end{Lemma}

\begin{proof}
The boundary principal symbol is given by the $(X,Y)$-Fourier transform of $K_{AB}$. In order to find a suitable $\chi$ to make $N_{AB}$ elliptic, we follow the strategy of \cite{UV15}, namely we do calculation for Gaussian function $\chi(s)=e^{-s^2/(2F^{-1}\alpha)}$ with $F>0$. Here $\chi$ does not have compact support, thus an approximation argument will be necessary at the end. The calculation of the Fourier transform of $K_{AB}$ with Gaussian like $\chi$ is similar to \cite[Lemma 4.1]{UV15} and \cite[Lemma 3.5]{SUV14}. It is not difficult to get that the boundary principal symbol of $N_{AB}$, $\sigma_{sc}(N_{AB})$, is a non-zero multiple of
\begin{equation*}
\int_{\mathbb S^{n-2}} \frac{1}{\sqrt{\xi^2+F^2}}\begin{pmatrix} B^* \\ -\frac{(\xi+iF)\hat Y\cdot\eta}{\xi^2+F^2} \\ \hat Y \end{pmatrix}A^*A\begin{pmatrix}B & -\frac{(\xi-iF)\hat Y\cdot\eta}{\xi^2+F^2} & \hat Y \end{pmatrix} e^{-|\hat Y\cdot \eta|^2/2F^{-1}\alpha(\xi^2+F^2)}\,d\hat Y.
\end{equation*}

Given any non-zero pair $[\varphi, \Phi]$, $\Phi=(\Phi^0,\Phi')$,
\begin{equation*}
\begin{split}
& (\sigma_{sc}(N_{AB}) [\varphi,\Phi],[\varphi,\Phi])\\
= & \frac{c}{\sqrt{\xi^2+F^2}}\int_{\mathbb S^{n-2}}\Big|A\Big(B\varphi-\frac{(\xi-iF)\hat Y\cdot\eta}{\xi^2+F^2} \Phi^0+\hat Y\cdot\Phi'\Big)\Big|^2e^{-|\hat Y\cdot \eta|^2/2F^{-1}\alpha(\xi^2+F^2)}\,d\hat Y.
\end{split}
\end{equation*}
Similar to Lemma \ref{fiber infinity}, to prove the ellipticity, it suffices to show that there is $\hat Y$ such that $B(0,y,0,\hat Y)\varphi-\frac{(\xi-iF)\hat Y\cdot\eta}{\xi^2+F^2} \Phi^0+\hat Y\cdot\Phi'\neq 0$. Again, we prove by contradiction, assume that for any $\hat Y\in \mathbb S^{n-2}$, $B(0,y,0,\hat Y)\varphi-\frac{(\xi-iF)\hat Y\cdot\eta}{\xi^2+F^2} \Phi^0+\hat Y\cdot\Phi'$ always vanishes. Then by the evenness of $B$, $B(0,y,0,\hat Y)\varphi+\frac{(\xi-iF)\hat Y\cdot\eta}{\xi^2+F^2} \Phi^0-\hat Y\cdot\Phi'=0$ too, which implies that $\varphi=0$ and $-\frac{(\xi-iF)\hat Y\cdot\eta}{\xi^2+F^2} \Phi^0+\hat Y\cdot\Phi'=0$ for all $\hat Y$. Similar to Lemma \ref{fiber infinity}, it is not difficult to see that there is $c=(c^1,\cdots,c^n)$ such that $\Phi^k=c^k(\xi+iF,\eta)$ for $k=1,\cdots,n$. Now by the fact that the diagonal of $\Phi=(\Phi^i_j)$ vanishes under the orthonormal basis at $z=(0,y)$, we conclude that $\Phi=0$, which is a contradiction. This establishes the ellipticity of $N_{AB}$ for Gaussian like $\chi$. Moreover, one can derive the following lower bound for the principal symbol
$$(\sigma_{sc}(N_{AB})[\varphi, \Phi],[\varphi, \Phi])\geq C\<(\xi,\eta)\>^{-1}\Big|[\varphi,\Phi]\Big|^2.$$

Finally by an approximation argument, one can show that there exists a compactly supported even function $\chi$ on $\mathbb R$, close to a Gaussian function, such that $N_{AB}$ defined by such $\chi$ is still elliptic as desired.
\end{proof}

Combining Lemma \ref{fiber infinity} and \ref{finite points} gives the following ellipticity result

\begin{Proposition}\label{N_AB elliptic}
For any $F>0$, given $\Omega$ a neighborhood of $U\cap M$ in $U$, there exists $\chi\in C^{\infty}_{c}(\mathbb R)$ such that 
$N_{AB}$ is elliptic in $\Omega$ acting on the space $\{[\varphi,\Phi]\in C^{\infty}(U)^n\times\, ^{sc}T^*U^n : \<\Phi(Pu),u\>_g=0, \forall u\in TU\}$.
\end{Proposition}

\subsection{Injectivity of $I_{AB}$}

Proposition \ref{N_AB elliptic} essentially proves the local invertibility of the operator $I_{AB}$ acting on the space 
$$\mathcal S=\{[\varphi,\Phi]\in C^{\infty}(U)^n\times T^*U^n: \supp [\varphi,\Phi]\subset M\cap U^{int}, \<\Phi(u),u\>_g=0,\, \forall u\in TU\}.$$
Given $\Phi\in T^*U^n$, we denote $^{sc}\Phi\in\, ^{sc}T^*U^n$ the expression of $\Phi$ under the scattering cotangent basis. We decompose $N_{AB}$ as $W^{-1}e^{-F/x}L_{AB}\circ I_{AB}e^{F/x}W$, and denote $H^{s,r}_{sc}(U)$ the {\it scattering Sobolev space} of order $(s,r)$ on $U$, which locally is just the standard weighted Sobolev space $H^{s,r}(\mathbb R^n)=\<z\>^{-r}H^s(\mathbb R^n)$, see \cite[Section 2]{UV15} for details. Then by Proposition \ref{N_AB elliptic}, given $[\varphi,\Phi]\in \mathcal S$,
\begin{align*}
\|W^{-1}e^{-F/x}[\varphi,\,^{sc}\Phi]\|_{H^{s,r}_{sc}} &\leq C\|N_{AB}W^{-1}e^{-F/x}[\varphi,\,^{sc}\Phi]\|_{H^{s+1,r}_{sc}}\\
&=C\|W^{-1}e^{-F/x}L_{AB}\circ I_{AB}[\varphi,\Phi]\|_{H^{s+1,r}_{sc}}.
\end{align*}
Notice that the elliptic estimates above have no error terms due to the localness nature of the problem, see also \cite[Section 2]{UV15}. Thus
%$$\|W^{-1}e^{-F/x}[\varphi,\Phi]\|\leq C\|I_{AB}[\varphi,\Phi]\|.$$
if $I_{AB}[\varphi,\Phi]=0$ for magnetic geodesics close to $S_p\p M$, we get $[\varphi,\,^{sc}\Phi]=0$, i.e. $[\varphi,\Phi]=0$ near $p$. Moreover, such determination is stable under usual Sobolev norms on any compact subset of $O=M\cap U^{int}$. In summary, we have the following result

\begin{Proposition}\label{injectivity of I_AB}
Let dim $M\geq 3$ and $\p M$ is strictly magnetic convex at $p\in\p M$. Given smooth invertible matrix function $A$ and $B$ on $SM$ with $B(z,v)=B(z,-v)$ for any $(z,v)\in SM$, there exists a smooth function $\tx$ near $p$ with $O_p=\{\tx>-c\}\cap\overline{M}$ for sufficiently small $c>0$, such that given $[\varphi,\Phi]\in L^2_{loc}(O_p)$ with $\<\Phi(u),u\>_g=0$ for any $u\in T O_p$, $[\varphi,\Phi]$ can be stably determined by the weighted magnetic ray transform $I_{AB}[\varphi,\Phi]$ restricted to $O_p$-local magnetic geodesics with the following stability estimates
$$\|[\varphi,\Phi]\|_{H^{s-1}(K)}\leq C\|I_{AB}[\varphi,\Phi]\|_{H^s(\mathcal M_{O_p})}$$
for any compact subset $K$ of $O_p$ and $s\geq 0$, if $[\varphi,\Phi]\in H^s_{loc}(O_p)$. Here $\mathcal M_{O_p}$ is the set of $O_p$-local magnetic geodesics.
%in the following sense: for $s\geq 0$, $[\varphi,\Phi]\in H^s(\overline O_p)$, the $H^{s-1}$ norm of $[\varphi,\Phi]$ restricted to any compact subset of $O_p$ is controlled by the $H^s$ norm of $I_{AB}[\varphi,\Phi]$ restricted to $O_p$-local magnetic geodesics.
\end{Proposition}
%---------------------------------------------------------------------------------------------------------------

%%%%%%%%%%%%%%%%%%%%%%%%%%%%%%%%%%%%%%%%%%%%%%%%%%%%%%%%%%%%%%%%%%%%%%%

\section{Injectivity of the non-linear rigidity problems}

We prove the main rigidity results for both the local and global non-linear problems in this section.

\subsection{Lens rigidity}

We first prove the local magnetic lens rigidity and its application in the global case.
\medskip

\noindent{\em Proof of Theorem \ref{local lens rigid}.} Since $(c^2g_0,\Omega)$ and $(\tilde c^2g_0,\tilde\Omega)$ have the same lens data near $S_p(\p M)$, by Proposition \ref{identity from lens data}, $I_{AB}[\varphi,\Phi](\gamma)=0$ for magnetic geodesics $\gamma$ close enough to the ones tangent to $\p M$ at $p$, as in identity \eqref{identity6}, with $\varphi=d(\ln\frac{\tilde c}{c})$, $\Phi=Y-\tilde Y$. Then Proposition \ref{injectivity of I_AB} implies that $d(\ln\frac{\tilde c}{c})$ and $Y-\tilde Y$ vanish in $M$ near $p$. Since $c=\tilde c$ near $p$ on $\p M$, which implies that $\ln\frac{\tilde c}{c}(z)=0$ for $z\in \p M$ near $p$, then we have actually $\ln\frac{\tilde c}{c}(z)=0$ for $z\in M$ near $p$, i.e. $c=\tilde c$ in $M$ near $p$. On the other hand, $c=\tilde c$ and $Y=\tilde Y$ in $M$ near $p$ imply that $\Omega=\tilde\Omega$ there, which proves the theorem. \qed
\medskip

\noindent{\em Proof of Theorem \ref{global lens rigid}.} The proof is essentially using the layer stripping argument as in \cite{SUV16}, however, since $\p M$ is strictly magnetic convex, it is not necessary to be a level set of the foliation $f$, and $\p M\cap \big(M\setminus f^{-1}((a,b])\big)$ might not be empty. A similar case was considered in \cite{PSUZ16} for ordinary geodesics. 

%By the assumptions of Theorem \ref{thm:maingloballinear} and Lemma \ref{lemma_convex_anydimension}, there are $a<b$ such that $f^{-1}([a,b])=M$ with non-vanishing $df$ in $\{f>a\}$ (thus for $a<t\leq b$, $f^{-1}(t)\cap M^{int}$ is a regular strictly convex hypersurface). Moreover, $f^{-1}(a)$ is a single point on $M$ (either in the interior or on the boundary). We denote $f^{-1}(a)$ by $z_0$, we first prove the following weaker version of Theorem \ref{thm:maingloballinear}, i.e. up to the point $z_0$. We define $\Omega$ as a subset of $\p M$, where $\Omega=\p M$ if $z_0$ is an interior point and $\Omega=\p M\setminus \{z_0\}$ if $z_0\in \p M$. 

%\begin{Lemma}\label{global up to z_0}
%There exists $v\in C^{\infty}(M\setminus \{z_0\}, \mathbb C^{N})$ with $v|_{\Omega}=0$ such that $h=(d+A+\Phi)v$ in $M\setminus \{z_0\}$. 
%\end{Lemma}

It is not difficult to see that $f(M)\subset (-\infty,b]$ and $f^{-1}(b)\subset \p M$ (if not, $M\setminus f^{-1}((a,b])$ can not have empty interior). Let 
\begin{align*}
\tau:=\inf\{t\leq b: \,\supp (c-\tilde c) \cup \supp (\Omega-\tilde \Omega) \subset f^{-1}((-\infty,t])\}.
\end{align*}
From now on we denote $\supp (c-\tilde c) \cup \supp (\Omega-\tilde \Omega)$ by $K$. Since $\supp(c-\tilde c)$ and $\supp(\Omega-\tilde\Omega)$ are compact, the infimum actually can be reached, i.e. $K \subset f^{-1}((-\infty,\tau])$ and $K\cap f^{-1}(\tau)\neq \emptyset$.
%Let
%$$\tau:=\inf\{t\leq b: \exists v_t\in C^{\infty}(\{f>t\}), v_t|_{\p M\cap \{f>t\}}=0, \,\mbox{s.t.}\, h=(d+A+\Phi)v_t\, \mbox{on}\, \{f>t\}\}.$$
%Note that $M\setminus K=\emptyset$ if $K=M$, $\p M\setminus K=\p M$ if $K\subset M\setminus \p M$. 
We claim that $\tau\leq a$. 

First, since $\p M$ is strictly magnetic convex with respect to both magnetic systems, by Theorem \ref{local lens rigid} there exists an open neighborhood $U$ of $\p M$ in $M$ such that $K\subset M\setminus U$ (it also implies that $\tau<b$). Assume $\tau>a$, since $c=\tilde c$ and $\Omega=\tilde\Omega$ in $f^{-1}((\tau,b])$, by taking the limit we have that $c=\tilde c$ and $\Omega=\tilde\Omega$ on $f^{-1}(\tau)$ too. Thus $f^{-1}(t)$ is strictly magnetic convex with respect to both systems for $t\geq \tau$. Now let $p\in K\cap f^{-1}(\tau)$, since $K\cap \p M=\emptyset$, $p\notin \p M$, i.e. $p$ is an interior point. By convexity for $t\geq \tau$, there exists an open neighborhood $U_p$ of $p$ such that for any $(z,v)\in \p_+S\big(f^{-1}((-\infty,\tau])\big)$ sufficiently close to $S_p\big(f^{-1}(\tau)\big)$, $\gamma_{z,v}$ and $\tilde\gamma_{z,v}$ stay in $U_p\cup f^{-1}((\tau,b])$ until they hit $\p M$ in positive and negative finite times. Moreover, since $L=\tilde L$, $\ell=\tilde\ell$ and $(c^2g_0,\Omega)=(\tilde c^2g_0,\tilde\Omega)$ in $f^{-1}([\tau,b])$, it is easy to see that $\gamma_{z,v}$ and $\tilde\gamma_{z,v}$ exit $f^{-1}((-\infty,\tau])$ in the same time (denoted by $s$), and $(\gamma_{z,v}(s),\dot\gamma_{z,v}(s))=(\tilde\gamma_{z,v}(s),\dot{\tilde\gamma}_{z,v}(s))\in \p_-S\big(f^{-1}((-\infty,\tau])\big)$, i.e. the lens data of $(c^2g_0,\Omega)$ and $(\tilde c^2g_0,\tilde\Omega)$ on $\p_+S\big(f^{-1}((-\infty,\tau])\big)$ are equal near $S_p\big(f^{-1}(\tau)\big)$. Then Theorem \ref{local lens rigid} implies that $c=\tilde c$ and $\Omega=\tilde\Omega$ near $p$. We apply the result to all $p \in K\cap f^{-1}(\tau)$ (note it is closed) to conclude that there exist some $\kappa<\tau$ such that $K\subset f^{-1}((-\infty,\kappa])$, which is a contradiction, so $\tau\leq a$.

Thus $K\subset M\setminus f^{-1}((a,b])$ with empty interior, by smoothness this implies that actually $K=\emptyset$ and $c=\tilde c$, $\Omega=\tilde\Omega$ on $M$.
\qed

\subsection{Boundary rigidity with travel time data}

Now we prove the local boundary rigidity result with the help of the travel time data.
\medskip

\noindent{\em Proof of Theorem \ref{local boundary rigid}.} Since $\mathbb A=\tilde{\mathbb A}$ near $(p,p)$, by the argument of \cite[Lemma 2.1]{DPSU07}, $c=\tilde c$, $\iota^*\alpha=\iota^*\tilde\alpha$ on $\p M$ near $p$. Moreover, $\iota^*\Omega-\iota^*\tilde\Omega=\iota^*d(\iota^*\alpha-\iota^*\tilde\alpha)=0$. Then \cite[Lemma 2.5]{DPSU07} implies that $L=\tilde L$. Now as mentioned in Remark \ref{ell and T}, $T=\tilde T$ near $(p,p)$ is equivalent to $\ell=\tilde\ell$ near $S_p\p M$, note that $c=\tilde c$ on $\p M$ near $p$. Thus the uniqueness result is an immediate consequence of Theorem \ref{local lens rigid}. 
\qed

%%%%%%%%%%%%%%%%%%%%%%%%%%%%%%%%%%%%%%%%%%%%%%%%%%%%%%%%%%%%%%

\section{Lens rigidity for general smooth curves} 

The method discussed in section 2 and 3 can be applied to the study of the lens rigidity problem for more general systems. We consider smooth parametrized curves $\gamma$, $|\dot{\gamma}|\neq 0$, satisfying the following equation
\begin{equation}\label{general curves 1}
\nabla_{\dot{\gamma}}\dot{\gamma}=G(\gamma, \dot{\gamma}),
\end{equation}
with $G(z, v)\in T_zM$ smooth on $TM$. $\gamma=\gamma_{z,v}$ depends smoothly on $(z,v)=(\gamma(0),\dot{\gamma}(0))$. We call the collection of such smooth curves, denoted by $\mathcal{G}$, {\it a general family of curves}. Note that if $G\equiv 0$, $\mathcal{G}$ is the set of usual geodesics; if $G$ is the Lorentz force corresponding to some magnetic field, then $\mathcal{G}$ consists of magnetic geodesics. It is easy to see that the generating vector field of $\mathcal G$ is
$$\mathfrak G(z,v)={\bf G}(z,v)+G^j(z,v)\frac{\p}{\p v^j}.$$
Generally curves $\gamma$ are not necessarily of constant speed (unless $\<G(\gamma,\dot{\gamma}),\dot{\gamma}\>=0$ along $\gamma$), and the dependence of $G$ on the second variable $v$ could be non-linear. %For the sake of simplicity, we assume $\gamma\in \mathcal{G}$ are parameterized by arclength, i.e. $|\dot\gamma|\equiv 1$ along $\gamma$ (this somehow is equivalent to $\<G(z,v),v\>_g=0$). %(one can always reparametrize the curve to make this happen, and we will see later that our method also works for curves with non-constant speed). 
%We consider the lens rigidity problem of a system $(M,g,G)$, i.e. associated with a general family of curves $\mathcal G$. %i.e. $(If)(\gamma), \gamma\in\mathcal{G}.$ X-ray transforms for general curves were studied in e.g. \cite{DPSU07, FSU08}.

Since $\gamma$ could be of non-unit speed, we should consider the scattering relation defined on $TM$, instead of $SM$. We define the following set 
$$\p_{\pm}TM\setminus 0:=\{(z,v)\in TM | z\in \p M, v\in T_zM, v\neq 0, \pm\<v,\nu(z)\>\geq 0\}.$$ 
For the sake of simplicity, we assume that any $\gamma\in \mathcal G$ satisfies the following property: if $(\gamma(s), \dot\gamma(s))\in \p_+TM\setminus 0$ for some $s$, then $|\dot\gamma(s)|=1$, i.e. $(\gamma(s),\dot\gamma(s))\in \p_+SM$. Thus we have the scattering relation with respect to $\mathcal G$
$$L: \p_+SM\to \p_-TM\setminus 0,$$ 
with $L(z,v)=(z',v')$ not necessarily belonging to $\p_-SM$. The travel time
$$\ell: \p_+SM\to \mathbb R\cup \infty,$$
$\ell(z,v)$ does not necessarily coincide with the length of $\gamma_{z,v}$ again.
Given $z\in \p M$, we say that $M$ (or $\partial M$) is {\it strictly convex} at $z$ with respect to $\mathcal{G}$ if
$$\Lambda(z,v)>\<G(z,v),\nu(z)\>_g$$ 
for any $v\in S_z\p M$. It is easy to see that the geometric meaning of our definition is similar to the usual convexity with respect to the metric (geodesics). It is also consistent with the definition of the magnetic convexity.

Now assume $M$ is strictly convex at $p\in \partial M$ with respect to $\mathcal{G}$, we consider the partial data lens rigidity problem near $p$ in a fixed conformal class. As mentioned above, generally $G(z,v)$ depends on $v$ non-linearly, it is not expectable that one can simultaneously recover both the conformal factor $c$ and the bundle map $G$. In this section, we assume that $G$ is given, we recover $c$ from the lens data. This is somehow similar to the case of the ordinary lens rigidity problem. %and our argument simplifies that of \cite{SUV16}. %Since we focus on curves with unit speed, we redefine $G$ on $TM\setminus \{(z,0):z\in M\}$ as 
%$$G(z,v)=|v|^2_gG^j(z,\frac{v}{|v|_g})\frac{\p}{\p z^j}.$$
%One can check that with such $G$, $\gamma_{z,v}(ct)=\gamma_{z,cv}(t)$ for $c>0$, thus we have the freedom to consider curves with constant speed other than $1$. 

Similar to the integral identity \eqref{identity5}, we have the following identity for general curves.

\begin{Proposition}
Given two systems $(c^2g_0,G)$ and $(\tilde c^2g_0,G)$, assume $L(\sigma)=\tilde L(\sigma)$, $\ell(\sigma)=\tilde\ell(\sigma)$ for some $\sigma\in \p_+SM$, then
\begin{equation}\label{general identity 1}
\int_0^{\ell(\sigma)}\frac{\p\tilde{\Xi^l}}{\p v^i}(\ell(\s)-s,X(s,\s))(2v^iv^j-g_0^{ij}|v|_{g_0}^2)\frac{\p h}{\p z^j}(X(s,\s))ds=0,
\end{equation}
where $h=ln(\tilde c/c)$.
\end{Proposition}

\noindent{\em Remark:} If one considers the general case where the bundle map is unknown, then the identity becomes
$$\int_0^{\ell(\sigma)}\frac{\p\tilde{\Xi^l}}{\p v^i}(\ell(\s)-s,X(s,\s))\Big((2v^iv^j-g_0^{ij}|v|_{g_0}^2)\frac{\p h}{\p z^j}+(G^i-\tilde G^i)\Big)(X(s,\s))ds=0,$$
with $G(z,v)=G^i(z,v)\frac{\p}{\p z^i},\,\tilde G(z,v)=\tilde G^i(z,v)\frac{\p}{\p z^i}$ for $(z,v)\in TM$.

Identity \eqref{general identity 1} induces the following weighted X-ray transform
\begin{equation}\label{general identity 2}
I_{AB}\varphi(\gamma)=\int A(\gamma(s),\dot{\gamma}(s))B(\gamma(s),\dot{\gamma}(s))\,\varphi(\gamma(s))\,ds=0,
\end{equation}
with invertible matrices $A$ and $B$. X-ray transforms along general smooth curves were considered in \cite{FSU08} and \cite[Appendix]{UV15}. Notice that $\gamma(s)$ does not necessarily have unit speed, by reparametrization, there exists a smooth positive function $\omega$ on $SM$ such that
\begin{equation}\label{general identity 3}
I_{AB}\varphi(\gamma)=\int A(\gamma(t),\omega\dot{\gamma}(t))B(\gamma(t),\omega\dot{\gamma}(t))\,\varphi(\gamma(t))\omega^{-1}(\gamma(t),\dot{\gamma}(t))\,dt=0,
\end{equation}
with $|\dot{\gamma}(t)|_g\equiv 1$.

Denote the product of the terms other than $\varphi$ in the integrand of \eqref{general identity 3} by $\mathcal W$, which is invertible on $SM$, thus
\begin{equation}\label{general identity 4}
I_{\mathcal W}\varphi(\gamma):=\int \mathcal W(\gamma(t),\dot{\gamma}(t))\,\varphi(\gamma(t))\,dt=0.
\end{equation}
Such matrix weighted X-ray transform was also considered in \cite{PSUZ16} for geodesic flows. By similar arguments as in section 3, we can prove that $I_{\mathcal W}$ is locally invertible. Thus the following weaker lens rigidity result holds for a general family of smooth curves.

\begin{Theorem}\label{general curves local lens rigid}
Let $n=$dim$M\geq 3$, let $c,\tc>0$ be smooth functions, $G:TM\to TM$ be a smooth bundle map and let $\p M$ be strictly convex with respect to both $(c^2g_0,G)$ and $(\tc^2g_0,G)$ near a fixed $p\in\p M$. Assume that on $\p M$ near $p$, $c,\tilde c$ has the same boundary jet. If $L=\tL$, $\ell=\tl$ near $S_p\p M$, then $c=\tc$ in $M$ near $p$.
\end{Theorem}

A global result can be derived also under proper foliation conditions. Notice that in Theorem \ref{general curves local lens rigid} we assume that the boundary jets of $c$ and $\tilde c$ are equal. It is unclear that whether one can make the assumption weaker, for example just assuming $c=\tilde c$ on $\p M$ near $p$. The usual proof of the determination of boundary jets relies on the existence of some `distance' minimizing functional (not necessarily equalling the travel time) on $\p M\times \p M$ which satisfies an eikonal equation. However given a general system, the existence of such minimizing functional is unknown. One might want to first consider it on a general family of curves whose induced flow is Hamiltonian.

%-------------------------------------------------------------------------------------------------------------------------------

\end{document}